\documentclass[reqno,a4paper]{amsart}
\usepackage{rotating}
\usepackage{amsthm}
\usepackage{mathdots}
\usepackage{graphicx}
\usepackage{amsmath}
\usepackage{mathrsfs}
\usepackage{amsfonts}
\usepackage{amssymb}
\usepackage{enumerate}
\usepackage{latexsym}
\usepackage{graphpap}
\usepackage{cite}
\usepackage{tikz}
\usepackage{bm}
\usepackage{multirow}
\usepackage{makecell}
\usepackage{blkarray}
\usepackage{amsmath}
\usepackage{caption}
\newtheorem{theorem}{Theorem}[section]

\newtheorem{lemma}[theorem]{Lemma}
\newtheorem{proposition}[theorem]{Proposition}

\theoremstyle{theorem}

\newtheorem{remark}[theorem]{Remark}

\theoremstyle{remark}
\numberwithin{equation}{section}

\newcommand{\dom}{\rm{dom}\,}
\newcommand{\im}{\rm{im}\,}

\begin{document}

\title[]{Probabilistic results for monoids of order-preserving transformations}

\thanks{This research was partially supported by the National Natural Science Foundation of China (Nos. 12271224, 12571018, 12401027), the Fundamental Research Funds for the Central University (No. lzujbky-2023-ey06) and Gansu Provincial Department of Education: Innovation Star Project for Graduate Students in Universities in Gansu Province (No. 2026CXZX-151)}

\author[Y. An]{Yang An}
\author[W. T. Zhang]{Wen Ting Zhang$^{\star}$}\thanks{$^\star$ Corresponding author} %
\address{$^{1}$ School of Mathematics and Statistics, Lanzhou University, Lanzhou, Gansu 730000, PR China} %
\email{any2024@lzu.edu.cn; zhangwt@lzu.edu.cn}



\begin{abstract}
Let $\mathcal{PO}_n$ be the monoid of all order-preserving partial transformations on $X_n=\{1,\dots, n\}$ with the natural order, and let $\mathcal{O}_n$ and $\mathcal{POI}_n$ denote its submonoids of order-preserving full and injective partial transformations, respectively. For each transformation $\alpha\in\mathcal{PO}_n$, write the random variables $Y(\alpha)=|{\im}\alpha|$ and $Y_r(\alpha)=|{\im}\alpha|$ given that $|{\dom}\alpha|=r$ for $0 \leqslant r \leqslant n$. We determine the probability distribution, expectation and variance of $Y_r$ and $Y$ for $\mathcal{PO}_n$ and $\mathcal{POI}_n$. In particular, $Y_r(\alpha)$ follows a hypergeometric distribution $H(n+r-1,n,r)$ for $\alpha \in \mathcal{PO}_n$, while $Y_r(\alpha)$ is degenerate and $Y(\alpha)$ follows a hypergeometric distribution $H(2n,n,n)$ for $\alpha \in \mathcal{POI}_n$.
\end{abstract}
\maketitle

\section{Introduction}
Let $X_n=\{1,\ldots,n\}$ be equipped with the natural order. Denote by $\mathcal{PT}_n$ the monoid of all partial transformations on $X_n$. A transformation $\alpha \in \mathcal{PT}_n$ is called \textit{order-preserving} if, for all $x, y \in \dom\alpha$, $x \leqslant y$ implies $x \alpha \leqslant y \alpha$. Let $\mathcal{PO}_n \subseteq \mathcal{PT}_n$ be the submonoid of order-preserving partial transformations. Its submonoids $\mathcal{O}_n$ and $\mathcal{POI}_n$ consist of order-preserving full transformations and of injective order-preserving partial transformations, respectively.

Combinatorial and probabilistic properties of transformation semigroups have been studied extensively, yielding many interesting results. For each $\alpha\in\mathcal{PO}_n$, define the random variables $Y(\alpha)=|{\im}\alpha|$ and $Y_r(\alpha)=|{\im}\alpha|$ given that $|{\dom}\alpha|=r$ for $0 \leqslant r \leqslant n$. Higgins explored the expectation and variance of $Y_n$ for the monoid $\mathcal{O}_n$ and proved that $ E(Y_n)=\frac{n^2}{2n-1}$, $\sigma^2(Y_n)=\frac{(n-1)n^2}{2(2n-1)^2}$ \cite{higgins1993}. Howie showed the combinatorial properties of $\mathcal{O}_n$ and determined that $|\mathcal{O}_n|=\binom{2n-1}{n-1}$ \cite{howie1971}. Subsequently, Gomes and Howie proved that $|\mathcal{PO}_n|=\sum_{r=0}^n\binom{n}{r}\binom{n+r-1}{r}$, the rank of $\mathcal{O}_n$ is $n$ and the rank of $\mathcal{PO}_n$ is $2n-1$ for $n\geqslant2$ \cite{howie1992}. Further, Laradji and Umar computed the cardinalities of some equivalence classes in $\mathcal{O}_n$ \cite{umar2006} and Umar studied the cardinalities of some equivalences on $\mathcal{PO}_n$ \cite{umar2014}. For the injective case, Garba studied the combinatorial properties of $\mathcal{POI}_n$ and showed that $|\mathcal{POI}_n|=\binom{2n}{n}$ \cite{garba1994}. Additionally, Laradji and Umar gave the formulae for the number of transformations $\alpha\in\mathcal{POI}_n$ having exactly $k$ fixed points \cite{umar2007}. In \cite{umar2010}, Umar collected various enumeration results and highlighted open problems for $\mathcal{POI}_n$.

In this paper, we study the probabilistic properties of $\mathcal{PO}_n$ and $\mathcal{POI}_n$. Note that if $\alpha\in\mathcal{PO}_n$ is a null transformation, $Y_r(\alpha)$ degenerates at $0$. Hence we always assume that $\alpha$ is not a null transformation in this paper. In Section \ref{Section2}, for each $\alpha\in\mathcal{PO}_n$, we determine the probability distribution, expectation and variance of $Y_r(\alpha)$ and $Y(\alpha)$, respectively. By these results, the probability distribution, expectation and variance of $Y_n(\alpha)$ and $Y(\alpha)$ for each $\alpha\in\mathcal{O}_n$ can be directly deduced. In Section \ref{Section3}, for each $\alpha\in\mathcal{POI}_n$, we prove that $Y_r(\alpha)$ is degenerate with $E(Y_r)=r$ and $\sigma^2(Y_r)=0$, and $Y(\alpha)$ follows a hypergeometric distribution $H(2n,n,n)$ with $E(Y)=\frac{n}{2}$ and $\sigma^2(Y)=\frac{n^2}{4(2n-1)}$.

Refer to the monograph by Howie \cite{howie1995} for any undefined notation and terminology of semigroup theory. We list some known combinatorial and probabilistic results that we need later.

\begin{proposition} For natural numbers $a, b, k, r$ and $n$, the following identities hold
\begin{align}
\sum_{k=0}^r\binom{a}{k}\binom{b}{r-k}
&=\binom{a+b}{r},\label{1}\\
\sum_{k=1}^r\binom{n}{k}\binom{r-1}{k-1}
&=\binom{n+r-1}{r},\label{2}\\
\sum_{k=1}^r k \binom{n}{k}\binom{r-1}{k-1}
&=n \binom{n+r-2}{r-1},\label{3}\\
\sum_{k=1}^r k^2 \binom{n}{k} \binom{r-1}{k-1}
&=n(n-1)\binom{n+r-3}{r-2}+n\binom{n+r-2}{r-1}.\label{4}
\end{align}
\end{proposition}

\begin{proof}
The identity (\ref{1}) is Vandermonde's Convolution Identity \cite{riordan1968}. The identities (\ref{2}), (\ref{3}) and (\ref{4}) follow from
\[
\sum_{k=1}^r\binom{n}{k}\binom{r-1}{k-1}\;=\;\sum_{k=0}^r\binom{n}{k}\binom{r-1}{k-1} \;\stackrel{(\ref{1})}{=}\;\binom{n+r-1}{r},
\]
\begin{align*}
\sum_{k=1}^r k \binom{n}{k}\binom{r-1}{k-1}
&\;= \;\sum_{k=1}^{r} n \binom{n-1}{k-1}\binom{r-1}{k-1} \;=\; n \sum_{k=0}^{r-1}\binom{n-1}{k}\binom{r-1}{k} \\
&\; \stackrel{(\ref{1})}{=} \; n \binom{n+r-2}{r-1},
\end{align*}
\begin{align*}
\sum_{k=1}^r k^2 \binom{n}{k} \binom{r-1}{k-1}
&\;=\; \sum_{k=1}^r k(k-1) \binom{n}{k} \binom{r-1}{k-1}+\sum_{k=1}^r k \binom{n}{k} \binom{r-1}{k-1} \\
&\;\stackrel{(\ref{3})}{=}\; n(n-1)\binom{n+r-3}{r-2}+n\binom{n+r-2}{r-1}.
\end{align*}
\end{proof}

\begin{proposition}{\rm\cite{kolmogorov2018}} For natural numbers $k, r$,
\begin{enumerate}[(i)]\label{probability}
\item The probability distribution $P(Y=k)=\sum_{r} P(Y=k|Z=r) \cdot P(Z=r)$;
\item The expectation $E(Y)=\sum_{r} E(Y|Z=r) \cdot P(Z=r)$;
\item The variance $\sigma^2(Y)=E(Y^2)-(E(Y))^2$.
\end{enumerate}
\end{proposition}

\section{The probabilistic properties of $\mathcal{PO}_n$}\label{Section2}

In this section, we study the probability distribution, expectation and variance of $Y_r(\alpha)$ and $Y(\alpha)$ for each $\alpha\in\mathcal{PO}_n$. Let $W_{\mathcal{PO}_n}^{r}=\{\alpha \in \mathcal{PO}_n:|{\dom}\alpha|=r\}$.

\begin{lemma} For natural number $1 \leqslant r \leqslant n$,
\[|W_{\mathcal{PO}_n}^{r}|=\binom{n}{r}\binom{n+r-1}{r}.\]
\end{lemma}

\begin{proof}
Let $S_r=\{m_1, \ldots, m_r\}$ with $1 \leqslant m_1<\cdots<m_r \leqslant n$. For each $\alpha \in W_{\mathcal{PO}_n}^{r}$, $\alpha$ is determined by the choice of $x_1, \ldots, x_r \in X_n$ such that $m_i \alpha=x_i$ for $1 \leqslant i \leqslant r$, where $1 \leqslant x_1 \leqslant \cdots \leqslant x_r \leqslant n$. Once ${\dom}\alpha$ is chosen, ${\im}\alpha$ is a subset of $X_n$ with $|{\im}\alpha|=k \leqslant r$. Suppose that ${\im}\alpha=\{i_1, \ldots, i_k\}$. Then let $P_\alpha=\{p_1, \ldots, p_k\}$ be the subset of $S_r$ such that $p_j$ is the greatest member of $S_r$ with $p_j \alpha=i_j$ for each $1 \leqslant j \leqslant k$. Since $\alpha$ is order-preserving, the number of choices for $P_\alpha$ and ${\im}\alpha$ are $\binom{r-1}{k-1}$ and $\binom{n}{k}$, respectively. Note that $\alpha$ is determined by the choice of the pair $({\im}\alpha, P_\alpha)$. Thus the number of $\alpha$ with $|{\im}\alpha|=k$ is $\binom{n}{k}\binom{r-1}{k-1}$ for $1 \leqslant k \leqslant r \leqslant n$.

Denote $\alpha\in\mathcal{PO}_n$ with $|{\dom}\alpha|=r$ and $|{\im}\alpha|=k$ by $\alpha(r,k)$. Since the number of choices for $S_r$ is $\binom{n}{r}$,
\[|\alpha(r,k)|=\binom{n}{r} \binom{n}{k} \binom{r-1}{k-1}.\]
Hence
\begin{align*}
|W_{\mathcal{PO}_n}^{r}| =&\sum_{k=1}^r|\alpha(r,k)|=\sum_{k=1}^r \binom{n}{r} \binom{n}{k}\binom{r-1}{k-1}\\
=&\binom{n}{r} \sum_{k=1}^r \binom{n}{k}\binom{r-1}{k-1}\stackrel{(\ref{2})}{=}\binom{n}{r}\binom{n+r-1}{r}.\qedhere
\end{align*}
\end{proof}

For each $\alpha\in\mathcal{PO}_n$, let $Z(\alpha)=|{\dom}\alpha|$ and $P(Y_r(\alpha)=k)$ be the conditional probability function $P(Y(\alpha)= k \mid Z(\alpha)=r)$.

\begin{theorem}\label{P(Y_r(PO_n))}
For each $\alpha\in\mathcal{PO}_n$, the probability function of $Y_r(\alpha)$ is
\[ P(Y_r(\alpha)=k)=\frac{\binom{n}{k}\binom{r-1}{k-1}}{\binom{n+r-1}{r}}\]
where $1 \leqslant k \leqslant r \leqslant n$.
\end{theorem}

\begin{proof}
The probability function of $Y_r(\alpha)$ is
\[P(Y_r(\alpha)=k) =\frac{|\alpha(r,k)|}{|W_{\mathcal{PO}_n}^{r}|}= \frac{\binom{n}{r} \binom{n}{k} \binom{r-1}{k-1}}{\binom{n}{r}\binom{n+r-1}{r}}= \frac{\binom{n}{k}\binom{r-1}{k-1}}{\binom{n+r-1}{r}}.\]
Clearly, $P(Y_r(\alpha)=k) \geqslant 0$ for $1 \leqslant k \leqslant r \leqslant n$ and
\[\sum_{k=1}^n P(Y_r(\alpha)=k) =\sum_{k=1}^n \frac{\binom{n}{k}\binom{r-1}{k-1}}{\binom{n+r-1}{r}}= \frac{\sum_{k=1}^r\binom{n}{k} \binom{r-1}{k-1}}{\binom{n+r-1}{r}}\stackrel{(\ref{2})}{=} \frac{\binom{n+r-1}{r}}{\binom{n+r-1}{r}}=1,\]
thus the nonnegativity and the normalization of $P(Y_r(\alpha))$ hold.
\end{proof}

Recall that if the random variable $X$ follows a hypergeometric distribution $H(N,K,t)$, then the probability function ${P}(X=k) = \frac{\binom{K}{k} \binom{N-K}{t-k}}{\binom{N}{t}}$. It follows from Theorem \ref{P(Y_r(PO_n))} that for each $\alpha \in \mathcal{PO}_n$, $Y_r(\alpha)$ follows a hypergeometric distribution if $N=n+r-1$, $K=n$, $t=r$. Hence the following result holds.

\begin{lemma}
For each $\alpha \in \mathcal{PO}_n$, the random variable $Y_r(\alpha)$ follows a hypergeometric distribution for $1 \leqslant r \leqslant n$ and is denoted by $Y_r(\alpha) \sim H(n+r-1,n,r)$.
\end{lemma}

\begin{theorem}\label{E(Y_r(PO_n))}
For each $\alpha \in \mathcal{PO}_n$ and $1 \leqslant r \leqslant n$, the expectation and variance of $Y_r(\alpha)$ are
\[E(Y_r)=\frac{nr}{n+r-1} \quad\text{and}\quad \sigma^2(Y_r)=\frac{nr(n-1)(r-1)}{(n+r-1)^2(n+r-2)}.\]
\end{theorem}

\begin{proof}
Recall that the expectation and variance of $X$ are $\frac{Kt}{N}$ and $\frac{Kt(N-K)(N-t)}{N^2(N-1)}$ for $X \sim H(N,K,t)$, respectively. Since $Y_r(\alpha) \sim H(n+r-1,n,r)$, it is routine to show that $E(Y_r)=\frac{nr}{n+r-1}$ and $\sigma^2(Y_r) = \frac{nr(n-1)(r-1)}{(n+r-1)^2(n+r-2)}$.
\end{proof}

Next, we can calculate the probability function of $Y(\alpha)$ by $P(Y_r(\alpha)=k)$.

\begin{theorem}\label{P(Y(PO_n))}
For each $\alpha\in\mathcal{PO}_n$, the probability function of $Y(\alpha)$ is
\[P(Y(\alpha)=k)=\frac{1}{S}\binom{n}{k}\sum_{r=1}^n\binom{n}{r}\binom{r-1}{k-1}\]
where $1 \leqslant k \leqslant n$ and $S =\sum_{r=1}^n\binom{n}{r}\binom{n+r-1}{r}$ is only related to $n$.
\end{theorem}

\begin{proof}
By Proposition \ref{probability}(i), the probability function of $Y(\alpha)$ is
\begin{align*}
P(Y(\alpha)=k) & =\sum_{r=1}^n P(Y_r(\alpha)=k) \cdot P(Z(\alpha)=r) =\sum_{r=1}^n\frac{\binom{n}{k}\binom{r-1}{k-1}}{\binom{n+r-1}{r}} \cdot \frac{|W_{\mathcal{PO}_n}^{r}|}{S}\\
& =\frac{1}{S}\sum_{r=1}^n\frac{\binom{n}{k}\binom{r-1}{k-1}\binom{n}{r}\binom{n+r-1}{r}}{\binom{n+r-1}{r}}=\frac{1}{S}\binom{n}{k}\sum_{r=1}^n\binom{n}{r}\binom{r-1}{k-1}.
\end{align*}
Clearly, $P(Y(\alpha)=k) \geqslant 0$ for $1 \leqslant k \leqslant n$ and
\begin{align*}
\sum_{k=1}^n P(Y(\alpha)=k) &=\sum_{k=1}^n\frac{1}{S}\binom{n}{k} \sum_{r=1}^n\binom{n}{r}\binom{r-1}{k-1}=\frac{1}{S} \sum_{r=1}^n\binom{n}{r} \sum_{k=1}^r\binom{n}{k}\binom{r-1}{k-1}\\
&\stackrel{(\ref{2})}{=}\frac{1}{S} \sum_{r=1}^n\binom{n}{r}\binom{n+r-1}{r}=1,
\end{align*}
thus the nonnegativity and the normalization of $P(Y(\alpha))$ hold.
\end{proof}

\begin{theorem}\label{E(Y(PO_n))}
For each $\alpha \in \mathcal{PO}_n$, the expectation and variance of $Y(\alpha)$ are
\[E(Y)=\frac{n}{S}\sum_{r=1}^n\binom{n}{r}\binom{n+r-2}{r-1},\]
\[\sigma^2(Y)=\frac{n}{S}\sum_{r=1}^n\binom{n}{r}((n-1)\binom{n+r-3}{r-2}+\binom{n+r-2}{r-1})-(\frac{n}{S}\sum_{r=1}^n\binom{n}{r}\binom{n+r-2}{r-1})^2\]
where $S=\sum_{r=1}^n\binom{n}{r}\binom{n+r-1}{r}$ is only related to $n$.
\end{theorem}

\begin{proof}
By Proposition \ref{probability}(ii) and Theorem \ref{E(Y_r(PO_n))}, the expectation of $Y(\alpha)$ is
\begin{align*}
E(Y) & =\sum_{r=1}^n E(Y_r) \cdot P(Z(\alpha)=r) =\sum_{r=1}^n \frac{nr}{n+r-1} \cdot \frac{|W_{\mathcal{PO}_n}^{r}|}{S} \\
& =\frac{n}{S} \sum_{r=1}^n\binom{n}{r}\binom{n+r-1}{r} \frac{r}{n+r-1} =\frac{n}{S} \sum_{r=1}^n\binom{n}{r} \frac{(n+r-1)!r}{r!(n-1)!(n+r-1)} \\
& =\frac{n}{S} \sum_{r=1}^n\binom{n}{r} \frac{(n+r-2)!}{(r-1)!(n-1)!} =\frac{n}{S} \sum_{r=1}^n\binom{n}{r}\binom{n+r-2}{r-1}.
\end{align*}
By Theorem \ref{P(Y(PO_n))}, the expectation of $Y^2(\alpha)$ is
\begin{align*}
E(Y^2) & \;=\; \sum_{k=1}^n k^2 \cdot P(Y=k) =\sum_{k=1}^n k^2 \cdot \frac{1}{S}\binom{n}{k}\sum_{r=1}^n\binom{n}{r}\binom{r-1}{k-1}\\
& \;=\; \frac{1}{S} \sum_{r=1}^n\binom{n}{r}\sum_{k=1}^r k^2 \binom{n}{k} \binom{r-1}{k-1} \\
&\; \stackrel{(\ref{4})}{=}\; \frac{1}{S} \sum_{r=1}^n\binom{n}{r}(n(n-1)\binom{n+r-3}{r-2}+n\binom{n+r-2}{r-1})\\
& \; =\;\frac{n}{S}\sum_{r=1}^n\binom{n}{r}((n-1)\binom{n+r-3}{r-2}+\binom{n+r-2}{r-1}).
\end{align*}
Hence by Proposition \ref{probability}(iii), the variance of $Y(\alpha)$ is
\begin{align*}
\;&\sigma^2(Y)
=E(Y^2)-(E(Y))^2\\
=\;&\frac{n}{S}\sum_{r=1}^n\binom{n}{r}((n-1)\binom{n+r-3}{r-2}+\binom{n+r-2}{r-1})-(\frac{n}{S}\sum_{r=1}^n\binom{n}{r}\binom{n+r-2}{r-1})^2.\qedhere
\end{align*}
\end{proof}

\begin{remark}
In Higgins{\rm\cite{higgins1993}}, it was shown that the expectation and variance of $Y_n(\alpha)$ for each $\alpha \in \mathcal{O}_n$ are $E(Y_n)=\frac{n^2}{2n-1}$ and $\sigma^2(Y_n)=\frac{(n-1)n^2}{2(2n-1)^2}$. Clearly these results can be obtained by Theorem \ref{E(Y_r(PO_n))} directly. Further, by Theorem \ref{E(Y(PO_n))} we get the expectation and variance of $Y(\alpha)$ for each $\alpha \in \mathcal{O}_n$ are $E(Y)=\frac{n^2}{2n-1}$ and $\sigma^2(Y)=\frac{(n-1)n^2}{2(2n-1)^2}$.
\end{remark}

\section{The probabilistic properties of $\mathcal{POI}_n$}\label{Section3}
In this section, we study the probability distribution, expectation and variance of $Y_r(\alpha)$ and $Y(\alpha)$ for each $\alpha\in\mathcal{POI}_n$. Let $W_{\mathcal{POI}_n}^{r}=\{\alpha \in \mathcal{POI}_n:|{\dom}\alpha|=r\}$.

\begin{lemma}For natural number $1 \leqslant r \leqslant n$,
\[|W_{\mathcal{POI}_n}^{r}|={\binom{n}{r}}^2.\]
\end{lemma}

\begin{proof}
Let $S_r=\{m_1, \ldots, m_r\}$ with $1 \leqslant m_1<\cdots<m_r \leqslant n$. For each $\alpha \in W_{\mathcal{POI}_n}^{r}$, $\alpha$ is determined by the choice of $x_1, \ldots, x_r \in X_n$ such that $m_i \alpha=x_i$ for $1 \leqslant x_1 < \cdots < x_r \leqslant n$, $1 \leqslant i \leqslant r$. Once ${\dom}\alpha$ is chosen, ${\im}\alpha$ is a subset of $X_n$ with $|{\im}\alpha|=k=r$. Since $\alpha$ is injective order-preserving, the number of choices for ${\im}\alpha$ is $\binom{n}{r}$. Note that $\alpha$ is determined by the choice of ${\im}\alpha$. Thus the number of $\alpha$ with $|{\im}\alpha|=k$ is $\binom{n}{r}$ for $k = r$.

Denote $\alpha\in\mathcal{POI}_n$ with $|{\dom}\alpha|=r$ and $|{\im}\alpha|=k$ by $\alpha(r,k)$. Since the number of choices for $S_r$ is $\binom{n}{r}$,
\[
|\alpha(r,k)|=
\begin{cases}
{\binom{n}{r}}^2, &\text{if}\;\;1 \leqslant k = r \leqslant n,\\
0, &\text{otherwise}.
\end{cases}
\]
Hence
\[|W_{\mathcal{POI}_n}^{r}|=\sum_{k=1}^r|\alpha(r,k)|
=\sum_{k=r}^r \binom{n}{r} \binom{n}{r} ={\binom{n}{r}}^2.\qedhere\]
\end{proof}

For each $\alpha\in\mathcal{POI}_n$, let $Z(\alpha)=|{\dom}\alpha|$ and $P(Y_r(\alpha)=k)$ be the conditional probability function $P(Y(\alpha)= k \mid Z(\alpha)=r)$.

\begin{theorem}\label{P(Y_r(POI_n))}
For each $\alpha \in \mathcal{POI}_n$, the probability function of $Y_r(\alpha)$ is
\[P(Y_r(\alpha)=k)=
\begin{cases}
1, &\text{if}\;\;0 \leqslant k = r \leqslant n,\\
0, &\text{otherwise}.
\end{cases}\]
\end{theorem}

\begin{proof}
The probability function of $Y_r(\alpha)$ is
\[P(Y_r(\alpha)=k)=\frac{|\alpha(r,k)|}{|W_{\mathcal{POI}_n}^{r}|}=
\begin{cases}
1,&\text{if}\;\; 1 \leqslant k = r \leqslant n,\\
0,&\text{otherwise}.
\end{cases}\]
If $r=0$, $|W_{\mathcal{POI}_n}^{r}|$ consists of only the null transformation and $Y_r(\alpha)$ degenerates at $0$. Hence $P(Y_r(\alpha)=k)=1$ for $0 \leqslant k = r \leqslant n$.
Clearly, $P(Y_r(\alpha)=k) \geqslant 0$ for $0 \leqslant k \leqslant r \leqslant n$ and
\[\sum_{k=0}^n P(Y_r(\alpha)=k) =\sum_{k=r} P(Y_r(\alpha)=k)=1,\] thus the nonnegativity and the normalization of $P(Y_r(\alpha))$ hold.
\end{proof}

Recall that if the random variable $X$ degenerates at $t$, then the probability function
${P}(X=k) =
\begin{cases}
1,&\text{if}\;\; k = t\\
0,&\text{otherwise}
\end{cases}$. It follows from Theorem \ref{P(Y_r(POI_n))} that for each $\alpha \in \mathcal{POI}_n$, $Y_r(\alpha)$ degenerates at $r$ for $0 \leqslant r \leqslant n$. Hence the following result holds.

\begin{lemma}
For each $\alpha \in \mathcal{POI}_n$, the random variable $Y_r(\alpha)$ follows a degenerate distribution for $0 \leqslant r \leqslant n$ and is denoted by $Y_r(\alpha) \sim \delta (r)$.
\end{lemma}

\begin{theorem}
For each $\alpha \in \mathcal{POI}_n$, the expectation and variance of $Y_r(\alpha)$ are $E(Y_r)=r$ and $\sigma^2(Y_r)=0$ respectively, where $0 \leqslant r \leqslant n$.
\end{theorem}

\begin{proof}
Recall the expectation and variance of $X$ are $t$ and $0$ for $X \sim \delta(t)$, respectively. Since $Y_r(\alpha) \sim \delta(r)$, it is routine to show that $E(Y_r)=r$ and $\sigma^2(Y_r) = 0$.
\end{proof}

Next, we can calculate the probability function of $Y(\alpha)$ by $P(Y_r(\alpha)=k)$.

\begin{theorem}\label{P(Y(POI_n))}
The probability function of $Y(\alpha)$ for each $\alpha \in \mathcal{POI}_n$ is
\[P(Y(\alpha)=k)=\frac{\binom{n}{k}^2}{\binom{2n}{n}}\]
where $0 \leqslant k \leqslant n$.
\end{theorem}

\begin{proof}
By Proposition \ref{probability}(i), the probability function of $Y(\alpha)$ is
\[\begin{aligned}
P(Y(\alpha)=k) & =\sum_{r=0}^n P(Y_r(\alpha)=k) \cdot P(Z(\alpha)=r)=P(Z(\alpha)=k)\\
&=\frac{|W_{\mathcal{POI}_n}^k|}{\sum_{r=0}^n|W_{\mathcal{POI}_n}^r|} =\frac{\binom{n}{k}^2}{\sum_{r=0}^n \binom{n}{r}^2}
=\frac{\binom{n}{k}^2}{\binom{2n}{n}}.
\end{aligned}\]
Clearly, $P(Y(\alpha)=k) \geqslant 0$ for $0 \leqslant k \leqslant n$ and
\[\sum_{k=0}^n P(Y(\alpha)=k)=\sum_{k=0}^n \frac{\binom{n}{k}^2}{\binom{2n}{n}}=\frac{1}{\binom{2n}{n}}\sum_{k=0}^n\binom{n}{k}^2=1,\]
thus the nonnegativity and the normalization of $P(Y(\alpha))$ hold.
\end{proof}

Recall that if the random variable $X$ follows a hypergeometric distribution $H(N,K,t)$, then the probability function ${P}(X=k) = \frac{\binom{K}{k} \binom{N-K}{t-k}}{\binom{N}{t}}$. It follows from Theorem \ref{P(Y(POI_n))} that for each $\alpha \in \mathcal{POI}_n$, $Y(\alpha)$ follows a hypergeometric distribution if $N=2n$, $K=n$, $t=n$. Hence the following result holds.

\begin{lemma}
For each $\alpha \in \mathcal{POI}_n$, the random variable $Y(\alpha)$ follows a hypergeometric distribution and is denoted by $Y(\alpha) \sim H(2n,n,n)$.
\end{lemma}

\begin{theorem}
For each $\alpha \in \mathcal{POI}_n$, the expectation and variance of $Y(\alpha)$ are
\[E(Y)=\frac{n}{2} \quad\text{and}\quad \sigma^2(Y)=\frac{n^2}{4(2n-1)}.\]
\end{theorem}

\begin{proof}
Recall that the expectation and variance of $X$ are $\frac{Kt}{N}$ and $\frac{Kt(N-K)(N-t)}{N^2(N-1)}$ for $X \sim H(N,K,t)$, respectively. Since $Y(\alpha) \sim H(2n,n,n)$, it is routine to show that $E(Y)=\frac{n}{2}$ and $\sigma^2(Y)=\frac{n^2}{4(2n-1)}$.
\end{proof}


\end{document}